\documentclass[a4paper,10pt]{article}
\usepackage{geometry}                
\usepackage[active]{srcltx}
\usepackage{hyperref}
\usepackage[normalem]{ulem}

\usepackage{amsmath}
\usepackage{amssymb}
\usepackage{parskip}
\usepackage{mathtools}
\usepackage{fullpage}

\usepackage{mathrsfs}
\usepackage{amsthm}
\usepackage[title]{appendix}
\usepackage{physics}
\usepackage{bbm}

\usepackage[utf8]{inputenc}
\usepackage{caption}
\usepackage{subcaption}
\usepackage{graphicx}
\usepackage{attachfile}
\usepackage{navigator}
\usepackage{theoremref}
\usepackage{faktor}
\usepackage{musicography}
\usepackage{accents}
\usepackage{tikz}
\usetikzlibrary{angles,quotes}

\newtheorem{theorem}{Theorem}[section]
\newtheorem{maintheorem}{Theorem}
\newtheorem{question}{Question}

\newtheorem{definition}[theorem]{Definition}
\newtheorem{proposition}[theorem]{Proposition}
\newtheorem{corollary}[theorem]{Corollary}
\newtheorem{lemma}[theorem]{Lemma}
\newtheorem{remark}[theorem]{Remark}

\newtheorem{assumption}[theorem]{Assumption}

\numberwithin{equation}{section}

\newcommand{\R}{\mathbb{R}}

\newcommand{\N}{\mathbb{N}}

\newcommand{\e}{\varepsilon}

\newcommand{\supp}{\mathrm{supp}}
\newcommand{\Lip}{\mathrm{Lip}}

\newcommand{\vol}{\mathrm{vol}}

\newcommand{\Ch}{\mathrm{Ch}}
\newcommand{\cd}{\mathsf{CD}}
\newcommand{\rcd}{\mathsf{RCD}}

\newcommand{\sfd}{\mathsf d}
\newcommand{\di}{{\rm d}}

\newcommand{\Capa}{\mathrm{Cap}}
\newcommand{\hh}{\mathcal{H}}

\newcommand{\LM}[1]{\hbox{\vrule width.2pt \vbox to#1pt{\vfill \hrule
width#1pt height.2pt}}}
\newcommand{\LL}{{\mathchoice
{\,\LM7\,}{\,\LM7\,}{\,\LM5\,}{\,\LM{3.35}\,}}}

\title{Infinitesimal Hilbertianity for manifolds with Sobolev regular Riemannian metrics through uniform lower capacity bounds}
\author{Vanessa Ryborz\thanks{University of Vienna. email: vanessa.emily.ryborz@univie.ac.at \\
ORCID: 0009-0009-4578-3620
}}
\date{}

\begin{document}

\maketitle
\begin{abstract}
 We prove that metric measure spaces arising from $2$-manifolds equipped with Riemannian metrics $g$ such that $g, g^{-1} \in L^\infty_{\rm loc} \cap W^{1,p}_{\rm loc}$ are infinitesimally Hilbertian. Along the way, we establish uniform lower $1$-capacity bounds of several $1$-dimensional subsets of $\R^2$.
 \let\thefootnote\relax
\end{abstract}

\tableofcontents

{\small \textbf{Acknowledgements.} 
The author gratefully acknowledges Andrea Mondino, Alexander Lytchak, Paul Schmitt, and Paul Paschmanns for valuable discussions. This research was funded in whole or in part by the Austrian Science Fund (FWF) 10.55776/STA32. For open access purposes, the author has applied a CC BY public copyright license to any author accepted manuscript version arising from this submission. 
}

\section{Introduction}
This paper can be seen as a continuation of the results obtained in \cite{ryborz2026infinitesimal}.

Manifolds with Riemannian metrics of low regularity (in particular below $C^2$) have been of growing interest and importance in the field of geometric analysis. 
They automatically enter the broader framework of metric measure spaces. A metric measure space is a triplet $(X, \sfd, \mathfrak{m})$, where $(X, \sfd)$ is a complete and separable metric space and $\mathfrak{m}$ is a $\sigma$-finite Borel measure on $X$ (playing the role of the volume measure). 

Several concepts from Riemannian geometry were generalised to metric measure spaces such as synthetic notions of the Ricci curvature bounded below by $K$ and the dimension bounded above by $N$ for some $K \in \R$ and $N \in [1, \infty]$; the curvature dimension condition $\cd(K, N)$ defined by Lott-Villani \cite{lott2009ricci, LottVillaniJFA} and Sturm \cite{sturm2006geometryI, sturm2006geometryII}.    

Our interest throughout this paper lies in the first order calculus on metric measure spaces, which in general do not provide an underlying differential structure.  After the initial approach by Cheeger \cite{cheeger1999differentiability}, followed by the approaches by Haj\l{}asz \cite{hajlasz1996sobolev}, and Shanmugalingam \cite{shanmugalingam2000newtonian}, Ambrosio, Gigli and Savaré introduced a Sobolev calculus based on the notion of test plans in \cite{ambrosio2014inventio} that they proved to be equivalent to the approaches in \cite{cheeger1999differentiability} and \cite{shanmugalingam2000newtonian}, leading to the definition of the synthetic Sobolev space $W^{1,2}(X)$, and a Laplace operator. 

In \cite{gigli2015differential}, Gigli studied conditions under which the Laplacian is linear and proposed the notion of \textit{infinitesimal Hilbertianity}. A metric measure space is infinitesimally Hilbertian, if the associated Sobolev space $W^{1,2}(X)$ is a Hilbert space, or equivalently if the Cheeger energy (a generalisation of the Dirichlet energy to metric measure spaces) is a quadratic form. This criterion serves to detect those metric measure spaces that resemble Riemannian manifolds on an infinitesimal scale, among possible Finslerian structures. 
 
Examples of infinitesimally Hilbertian metric measure spaces include geodesically complete weighted Riemannian manifolds \cite{luvcic2020infinitesimal}, locally $\mathrm{CAT}(\kappa)$-spaces equipped with any Borel measure  \cite{di2021infinitesimal}, and sub-Riemannian manifolds equipped with any Borel measure \cite{le2023universal}.
In \cite{ambrosio2014duke}, Ambrosio, Gigli, and Savaré proposed the Riemannian curvature dimension condition, $\rcd(K, \infty)$, a refinement of the curvature dimension condition which combines the $\cd(K, \infty)$-condition with infinitesimal Hilbertianity.
The $\rcd(K, N)$-condition for $N<\infty$ was proposed by Gigli in \cite{gigli2014overview}. 
Using this terminology, the metric measure space $(\R^n, \norm{\cdot}, \mathcal{L}^n)$ is an $\rcd(0, N)$-space if and only if $\norm{\cdot}$ arises from an inner product, hence in the setting of the Euclidean space infinitesimal Hilbertianity indeed detects Riemannian structures among Finslerian ones.

The main objective of this paper is to compare the synthetic first order calculus on metric measure spaces with the classical first order calculus on manifolds with non-smooth Riemannian metrics. In the case of a smooth manifold $M$ equipped with a continuous Riemannian metric $g$ such comparison questions were previously studied in \cite{burtscher2012length, mondino2025equivalence}.

Various relevant classes of Riemannian metrics are non-continuous, including for instance for instance the Geroch-Traschen class of metrics \cite{geroch1987strings}, i.e. those Riemannian metrics $g$ such that $g, g^{-1} \in L^\infty_{\rm loc}(M)$ and $g \in W^{1,2}_{\rm loc}(M)$. These regularity assumptions are minimal to compute the distributional Ricci curvature on $M$ (see \cite{grosser2013geometric}, \cite{graf2020singularity}) and have been studied in both Riemannian (see for example \cite{grant2011synthetic}) and Lorentzian signature (see \cite{steinbauer2009geroch}). Moreover, in the context of potential theory for uniformly elliptic operators, the even broader class of (Riemannian) metrics $g$ such that $g, g^{-1} \in L^\infty_{\rm loc}(M)$ has been studied in \cite{saloff1992uniformly}, \cite{sturm1995analysis}, \cite{norris1997heat} and \cite{de1991integral}. 

If the metric tensor is merely measurable, we work with the Riemannain distance $\sfd_g$ defined by Norris (cf. \cite{norris1997heat}) and De Cecco-Palmieri (cf. \cite{de1991integral}) defined the Riemannian distance $\sfd_g$ for Riemannian metrics $g$ that satisfy $g, g^{-1} \in L^\infty_{\rm loc}$ (see Subsection \ref{sec:3}). Their definition is consistent with the definition of the Riemannian distance induced by continuous Riemannian metric tensors, but avoids the difficulty that the image of a continuous curve is a null set (which prevents us from using the usual definition of the Riemannian distance).

To study infinitesimal Hilbertianity, a central quantity is the metric speed of absolutely continuous curves (\cite{ambrosio1990metric, kirchheim1994rectifiable}). Consequently, we will focus on the following two comparison questions on manifolds $M$ equipped with non-smooth Riemannian metrics $g$.
\smallskip

\begin{question}[see also Question 2 in \cite{ryborz2026infinitesimal}]\label{q1}
   Under which regularity assumptions on $g$ does the metric speed of ``almost every" absolutely continuous curves equal the norm of its derivative, i.e.  $\lim_{h \to 0} \frac{\sfd_g(\gamma_t, \gamma_{t+h})}{|h|}=|\dot{\gamma}_t|_g$ for almost every $t \in [0,1]$?
\end{question}
Here, ``almost every" should be understood in the sense of test plans, see  Subsection \ref{sec:2}.
\smallskip

\begin{question}[see also Question 3 in \cite{ryborz2026infinitesimal}]\label{q2}
   Is the metric measure space $(M, \sfd_g, \vol_g)$ infinitesimally Hilbertian?
\end{question}
Related questions were posed by De Giorgi in \cite{de1990conversazioni}.

The results in \cite{burtscher2012length, mondino2025equivalence} show that if the Riemannian metric $g$ is continuous, then both Question \ref{q1} and \ref{q2} have positive answer.

If $g$ merely satisfies $g, g^{-1}\in L^\infty_{\rm loc}$, then Sturm establishes, that there exist cases, where the answer to Question \ref{q1}, if one does not assume the metric $g$ to be continuous \cite{sturm1997diffusion}. 
Moreover, in \cite{de1995lip}, De Cecco and Palmieri present one example of a metric $g$ on $\R^2$ with $g, g^{-1} \in L^\infty$ and such that on a positive measure set of points, the tangent is a normed space, with a norm that does not arise from an inner product hinting that Question \ref{q2} may have negative answer under these broad regularity assumptions.
 
The examples in \cite{ryborz2026infinitesimal} show that there exist manifolds of dimension $d\geq 3$, and Riemannian metrics $g$ on $M$ such that $g, g^{-1} \in L^\infty\cap W^{1,d-1}_{\rm loc}$ (and hence $W^{1, p}_{\rm loc}$ for $p \in [1, d-1]$) for which both Question \ref{q1} and \ref{q2} have negative answer. Moreover, \cite{ryborz2026infinitesimal} establishes that there exist $2$-manifolds with Riemannian metrics $g$ such that $g, g^{-1} \in L^\infty$ and such that Question \ref{q2} has negative answer.
Local Sobolev regularity can be seen as a bridge from $L^\infty$ to $C^0$, hence a natural idea is to study whether the Sobolev regularity is a discriminant for the above questions.

The main result of this paper is:
\smallskip

\begin{maintheorem}[see Theorem \ref{2d_infinitesimal_hilbertianity}]
    Let M be a $2$-manifold and $g$ be a Riemannian metric on $M$ such that $g, g^{-1} \in L^\infty_{\rm loc}\cap W^{1,p}_{\rm loc}$ for some $p \geq 1$. Then $(M, \sfd_g, \vol_g)$ is infinitesimally Hilbertian.  
\end{maintheorem}

In the recent work of Lytchak and Soultanis \cite{lytchak2026infinitesimally} the authors establish the equivalence of infinitesimal Hilbertianity and the property (ET) on metric length surfaces.  

\textbf{Outline of the paper and strategy.}
In Section \ref{sec:cap}, we will study the Sobolev $p$-capacity as from \cite{evans2025measure}. In Subsection \ref{subsec:def_cap} we review the definition and some basic properties of the $p$-capacity. Thereafter, in Subsection \ref{subsec:cap_bound}, we prove uniform lower $1$-capacity bounds for bi-Lipschitz images of one dimensional sets into $\R^2$.
In Section \ref{sec:application}, we apply said bounds to compare the synthetic and classical first order calculus on  $2$-manifolds with non-smooth Riemannian metrics.  
In Subsection \ref{sec:2}, we will recall some basic notions on the first order calculus on metric measure spaces. In Section \ref{sec:3}, we recall the definition of the distance in the case of a non-continuous Riemannian metric by Norris and De Cecco-Palmieri and explain how a manifold $M$ equipped with a non-continuous Riemannian metric $g$ enters the metric measure space setting. Thereafter, we briefly study absolutely continuous curves and test plans in that setting. 
In Subsection \ref{sec:4}, we will use the results from Section \ref{sec:cap} to establish positive answers to Questions \ref{q1} and \ref{q2} if $M$ is two-dimensional and $g$ satisfies $g, g^{-1} \in L^\infty_{\rm loc}\cap W^{1,p}_{\rm loc}$. Question \ref{q1} is addressed in Proposition \ref{2d_metric_speed_sobolev} and Question \ref{q2} is addressed in Theorem \ref{2d_infinitesimal_hilbertianity}.
\smallskip

\subsection*{Notation}
Given a manifold $M$, a Riemannian metric $g$, and a locally Lipschitz continuous function $f:M \to \R$, $|Df|_w$ will always denote its minimal weak upper gradient, $\nabla_gf$ its gradient with respect to $g$. Given a (weakly) differentiable function $f:\R^d \to \R$, we denote by $Df= \nabla_{euc} f$ the (weak) Euclidean gradient. \\
Given an absolutely continuous curve $\gamma:[0,1] \to M$, we will for almost every $t \in [0,1]$ denote its metric speed at time $t$ by $|\dot{\gamma}_t|$, and its derivative at time $t$ by $\dot{\gamma}_t \in T_{\gamma_t}M$. To avoid confusion, if we take the norm of the derivative, we will always write $|\dot{\gamma}_t|_g$ for $\sqrt{g(\dot{\gamma}_t,\dot{\gamma}_t)}$ and $|\dot{\gamma}_t|_{euc}$ for $\sqrt{\langle\dot{\gamma}_t,\dot{\gamma}_t\rangle_{euc}}$ (given specified local coordinates). \\

\subsection*{Disclaimer}
This is a rather rough version of this work, which I decided to upload at this stage with the aim to revise and polish the manuscript further in the future. I recently noticed a paper by Zhong \cite{zhong2026scoop} working on the same topic which motivated me to post my independent non-AI supported work sooner rather than later.
\section{Capacity bounds}\label{sec:cap}
\subsection{Fine properties of Sobolev functions}\label{subsec:def_cap}
In the following, we recall some of the theory on the capacity and its relationship with Sobolev functions, as it is presented in \cite[Chapter 4]{evans2025measure}.
\begin{definition}
    For $d \in \N$, $1 \leq p<d$, define the set 
    \begin{align*}
        K^p:= \{f: \R^d \to \R|\ f \geq 0, f \in L^{\frac{pd}{d-p}}(\R^d), Df \in L^p(\R^d, \R^d)\}.
    \end{align*}
    Given a set $A \subset \R^d$, define the $p$-capacity as
    \begin{align*}
        \mathrm{Cap}_p(A):= \inf\Big\{ \int_{\R^n} |Df|^p\, \di x \big|\ f \in K^p, A \subset \{f \geq 1\}^\circ\Big\},
    \end{align*}
    where $\{f \geq 1\}^\circ$ denotes the interior of the set $\{f \geq 1\}$.
\end{definition}
By \cite[Theorem 4.14]{evans2025measure}, $\Capa_p$ is an outer measure on $\R^d$. 
\begin{theorem}[\cite{evans2025measure}, Theorem 4.15]\label{basic_properties_capa}
    Assume $A \subset \R^d$. Then
    \begin{itemize}
        \item[(i)] $\Capa_p(A) = \inf\{\Capa_p(U)|\ U\ \mathrm{open,\ } A \subset U\}$. 
        \item[(ii)] $\Capa_p(\lambda A)= \lambda^{d-p} \Capa_p(A)$  for $\lambda>0$.
        \item[(iii)] $\Capa_p(LA)= \Capa_p(A)$ for any affine isometry $L: \R^d \to \R^d$.
        \item[(iv)] $\Capa_p(A)\leq C \hh^{d-p}(A)$ for some constant $C$ depending only on $d$ and $p$. 
        \item[(v)] $\mathcal{L}^d(A) \leq C\Capa_p(A)^{\frac{d}{d-p}}$ for some constant $C$ depending only on $d$ and $p$. 
    \end{itemize}
\end{theorem}
\smallskip


\smallskip

\begin{theorem}[\cite{evans2025measure}, Theorem 5.12]
    Let $d \geq 2$ and $A \subset \R^d$ be compact. Then $\Capa_1(A) = 0$ if and only if $\hh^{n-1}(A)=0$.
\end{theorem}
\smallskip

\begin{definition}[\cite{evans2025measure}, Definition 4.11]
    A function $f: \R^d \to \R$ is $p$-quasicontinuous if for each $\e>0$, there exists an open set $V$ such that 
    \begin{align*}
        \Capa_p(V) \leq \e
    \end{align*}
    and
    \begin{align*}
        f|_{\R^d \setminus V}\ \mathrm{is\ continuous.}
    \end{align*}
\end{definition}
\smallskip

\begin{theorem}[\cite{evans2025measure}, Theorem 4.19]\label{quasicontinuity}
    Suppose $f \in W^{1,p}(\R^d)$, $1 \leq p <d$. Then there exists a Borel set $E \subset \R^d$ such that $\Capa_p(E) =0$ and for all $x \in \R^d \setminus E$, it holds
    \begin{align*}
        \lim_{r \to 0}\frac{1}{|B_r(x)|} \int_{B_r(x)} f(y) \di y =: f^*(x)
    \end{align*}
    exists. Moreover $f^*$ is $p$-quasicontinuous.
\end{theorem}
\smallskip

In particular, the proof of \cite[Theorem 4.19]{evans2025measure} yields the following: Given a sequence of smooth functions $f_n$ that converge to $f$ in $W^{1,p}(\R^d)$ such that $\int |Df_n-Df|^p \, \di x \leq \frac{1}{2^{(p+1)n}}$, and define the sets $B_n \subset \R^d$ via 
\begin{align*}
    B_n:= \left\{ x \in \R^d \Bigg| \frac{1}{|B_r(x)|}\int_{B_r(x)} |f-f_n| \, \di y \geq \frac{1}{2^n}\ \mathrm{for \ some\ } r >0\right\}.
\end{align*}
Then, there exists a set $E \subset \R^d$ with $\Capa_p(E) = \mathcal{L}^d(E) =0$ such that all $x \in \R^d \setminus E$ are Lebesgue points. 
Now define the sets $E_k:= E\cup  \bigcup_{j \geq k} B_j$. Then the sequence $f_n$ converges uniformly on $\R^d\setminus E_k$ for all $k \in \N$ and $\lim_{k \to \infty} \Capa_p(E_k) =0$. 
\subsection{The case $d=2$, $p=1$.}\label{subsec:cap_bound}
In this section, we will study the $1$-capacity in $\R^2$.
\begin{lemma}\label{twodim_rotation_gradient}
   Let $L>0$, $b>0$. Let $f\in C^{0,1}([0, L], [0,b])$ with $f(0) =b$ and such that there exists an $s \in [0,L]$ such that $f(s) =0$. Then one has that 
    \begin{align*}
        \int_{0}^L |f'| \di t \geq b.
    \end{align*}
\end{lemma}
\begin{proof}
    This follows from the fundamental theorem of calculus as
    \begin{align*}
        b= f(0) = \int_0^s -f'\, \di t \leq \int_{0}^L |f'|\, \di t.
    \end{align*}
\end{proof}

\begin{lemma}\label{capacity_integral_estimate_one_line2d}
    Let $L >1$, $\lambda \in (0,L)$, $b>0$. Let $\gamma:[0,1] \to \R^2$ be a continuous curve such that $\gamma([0,1]) \subset B^{\R^2}_{4L}(0)$, $\gamma_0, \gamma_1 \in B^{\R^2}_L(0)$ and $|\gamma_0-\gamma_1|_{\R^2} \geq \lambda$. Now let $f \in C_c^{0,1}(\R^2, [0,1])$ such that $f \equiv b$ on $\gamma([0,1])$. Then
    \begin{align*}
        \int_{B_{4L}(0)} |f| + |Df| \di x \geq b\lambda.
    \end{align*}
\end{lemma}
\begin{proof}
     By potentially applying an orthogonal linear transformation, we may assume that $\gamma_{1}-\gamma_{0} \in \mathrm{span}(e_1)$. 
    Denote by $P_1: \R^2 \to \R$ the projection on the first coordinate. 
    By Fubini's theorem, we have that 
    \begin{align}\label{first_fubini2d}
        \int_{B_{4L}(0)} |f| + |Df| \di x \geq \int_{[-L, L]} \int_{\{y\}\times [-\sqrt{16L^2-y^2},\sqrt{16L^2-y^2}]} |f| + |Df|\,\di \mathcal{L}^{1}(x) \di \mathcal{L}^1(y).
    \end{align}
    Denote $y_1:= \min(P_1(\gamma_0), P_1(\gamma_1))$.
    Now, by the intermediate value theorem and the choice of $f$, we have that for each $y \in [y_1, y_1 + \lambda] \subset [-L, L]$ there exists a $x_y \in \big[-\sqrt{16L^2-y^2},\sqrt{16L^2-y^2}\big]$ such $(y, x_y) = \gamma_t$ for some $t \in [0,1]$ and hence $f((y, x_y)) =b$. 
    Now fix a $y \in [y_1, y_1 + \lambda]$. As $|y|\leq L$, we get that there exists a $v\in \{-1, +1\}$ such that $\{ x_y\}+[0, L]v \subset [-\sqrt{16L^2-y^2},\sqrt{16L^2-y^2}]$. 
    Then, 
    \begin{align}\label{2dslice_integral}
        \int_{\{y\}\times [-\sqrt{16L^2-y^2},\sqrt{16L^2-y^2}]} |f| + |Df|\, \di x \geq \int_{\{y\}\times (\{ x_y\}+[0, L]v)} |f| + |Df|\,\di x.
    \end{align}
    The last integral can be rewritten as
    \begin{align*}
       \int_0^L (|f| + |Df|)((y,x_y)+ \rho (0,v))\, \di \rho.
    \end{align*}
    Denote $f_v\colon [0,L] \to [0,b],\ \rho \mapsto f((y,x_y)+ \rho (0,v))$. Denote $\sigma := \min_{\rho \in [0,L]} f_v(\rho)$. 
    Then, we have that 
    \begin{align}\label{function_integral_on_line2d}
        \int_0^L  |f| ((y,x_y)+ \rho (0,v))\, \di \rho = \int_0^L|f_v| \di \rho \geq L\sigma \geq \sigma.
    \end{align}
    On the other hand, Lemma \ref{twodim_rotation_gradient} applied to $\max(f_v-\sigma, 0)$ yields that  
    \begin{align}\label{gradient_integral_on_line2d}
        \int_0^L |Df| ((y,x_y)+ \rho (0,v))\, \di \rho = \int_0^L |f_v'|\, \di \rho \geq b-\sigma.
    \end{align}
    Summing up \eqref{function_integral_on_line2d} and \eqref{gradient_integral_on_line2d} gives
    \begin{align}
        \int_0^L  (|f| + |Df|)((y,x_y)+ \rho (0,v))\, \di \rho \geq b. 
    \end{align}
    Combining this with \eqref{first_fubini2d} and \eqref{2dslice_integral}, we get that 
    \begin{align*}
        \int_{B_{4L}(0)} |f| + |Df| \di x \geq b\lambda.
    \end{align*}
    This proves the lemma. 
\end{proof}

Let $A \subset \R$ and $L>1$. We say that a map $F:A \to \R^d$ is $(\frac{1}{L}, L)$-bi-Lipschitz if for each $x, y \in A$ it holds $\frac{1}{L}|x-y| \leq |F(x)-F(y)| \leq L|x-y|$.
\smallskip

\begin{lemma}\label{2dcapacity_of_line_segments} 
    Let $L >1$. Let $K \subset [0,1]$ be the finite union of pairwise disjoint compact intervals $I_1, \ldots, I_n$ and let $\gamma:K \to B^{\R^2}_L(0)$ be a $(\frac{1}{L}, L)$-bi-Lipschitz continuous map. Now let $f \in C_c^{0,1}(\R^2, [0,1])$ such that $f \equiv 1$ on $\gamma(K)$. Then it holds 
    \begin{align*}
        \int_{B_{4L}(0)}|f|+|Df| \di x \geq \frac{\mathcal{L}^1(K)}{L}.
    \end{align*}
\end{lemma}
\begin{proof}
    We proceed by induction. \\
    \textbf{Step 1.} $n=1$. \\
    We have that $\gamma(K) = \gamma(I_1) \subset B_L(0)$ and $|\gamma_{\min I_1}-\gamma_{\max I_1}|_{\R^2} \geq \frac{\mathcal{L}^1(I_1)}{L}$. Then $\gamma$ is a curve and the case $n=1$ now follows from Lemma \ref{capacity_integral_estimate_one_line2d}, after reparametrising $\gamma$ and setting $b=1$. \\
    \textbf{Step 2.} $n>1$. \\
    Assume the property holds true for all $m \leq n-1$. Define 
    \begin{align*}
        a:= \sup\{s \in [0,1]: \gamma(I_1),\gamma(I_n)\ \mathrm{are\ in\ the\ same \ connected\ component\ of\ } f^{-1}((s, \infty)) \cap \overline{B_{4L}(0)}\}.
    \end{align*}
    In the following all statements about openness and compactness will be made with respect to the Euclidean subspace topology on $\overline{B_{4L}(0)}$.
    If $a>0$, pick an $\e \in (0, \min(a, 1-a)/4)$.
    Then $\gamma(I_1)$ and $\gamma(I_n)$ are path connected in $f^{-1}(a-\e, \infty)\cap \overline{B_{4L}(0)}$. Hence, there exists a continuous curve $\Tilde{\gamma}:[0,1] \to f^{-1}((a-\e, \infty))\cap \overline{B_{4L}(0)}$ such that $\Tilde{\gamma}_0 = \gamma_{\min {I_1}}$ and $\Tilde{\gamma}_1 = \gamma_{\max {I_n}}$. Note that $\gamma_{\min {I_1}}, \gamma_{\max {I_n}} \in B_L(0)$ and 
    \begin{align}
        |\gamma_{\min {I_1}}-\gamma_{\max I_n}|_{\R^2} \geq \frac{1}{L}\mathcal{L}^1(K).
    \end{align}
    Then Lemma \ref{capacity_integral_estimate_one_line2d} shows that for $\Tilde{f}:= \min(f, a-\e)$, it holds
    \begin{align}
        \int_{B_{4L}(0)}|\Tilde{f}|+|D\Tilde{f}| \di x \geq \frac{(a-\e)\mathcal{L}^1(K)}{L}. 
    \end{align}
    If $a=0$, let $\e \in (0, \frac{1}{4})$ we set $\Tilde{f} \equiv 0$, so in total, we get 
    \begin{align}
        \int_{B_{4L}(0)}|\Tilde{f}|+|D\Tilde{f}| \di x \geq \frac{\max(0,a-\e)\mathcal{L}^1(K)}{L}. 
    \end{align}
    Now consider $\check{f}:= \max(f, a+2\e)-a-2\e$. Then $\supp\, D\Tilde{f} \cap \supp\, D \check{f} = \emptyset$. Now there exists $j \in \{2, \ldots, n\}$, pairwise disjoint nonempty sets $B_1, \ldots, B_j \subset \{1, \ldots, n\}$, and pairwise disjoint open sets $O_1, \ldots, O_j \subset f^{-1}((a+\e, \infty)) \cap \overline{B_{4L}(0)}$ such that for each $k$, $O_k$ is a connected component of $f^{-1}((a+\e, \infty))$ and
    \begin{align}
        \bigcup_{i \in B_k} \gamma(I_i) \subset O_k.
    \end{align}
    Fix a $k \in \{1, \ldots, j\}$. Define
    \begin{align}
        \check{f}_k:= \check{f}\mathbbm{1}_{O_k}. 
    \end{align}
    Note that 
    \begin{align*}
        \check{f}_k \in C^{0,1}_c(B_{4L}(0)),
    \end{align*}
    as $f^{-1}([a+ \frac{3\e}{2}, \infty)) \cap \overline{B_{4L}(0)}$ is compactly contained in $f^{-1}((a+\e, \infty))\cap O_k \cap \overline{B_{4L}(0)}$, so for each $k$ there exists a function $ \eta_k\in  C_c^\infty(O_k\cap\overline{B_{4L}(0))}, [0,1])$ such that $\eta_k \equiv 1$ on $f^{-1}([a+ \frac{3\e}{2}, \infty)) \cap \overline{B_{4L}(0)}$, hence $\check{f}_k = \eta_k \check{f}$.
    
    Denote $J_k:= \bigcup_{i \in B_k} I_i \subset [0,1]$ and $\gamma_k:= \gamma|_{J_k}$. Then $\gamma_{k}$ is $(\frac{1}{L}, L)$-bi-Lipschitz. 
    Now as $|B_k|< n$, rescaling the induction hypothesis yields that 
    \begin{align}
        \int_{B_{4L}(0)}|\check{f}_k|+|D\check{f}_k| \di x \geq \frac{\mathcal{L}^1(J_k)}{L} (1-a-2\e). 
    \end{align}
    Using that $f \geq \check{f}+ \Tilde{f}$, as well as $\check{f}, \Tilde{f}\geq 0$, and for $k \neq k'$, it holds $\supp\, \check{f}_{k'} \cap \supp \check{f}_k =\emptyset$, we get that 
    \begin{align*}
         \int_{B_{4L}(0)}|f|+|Df| \di x &\geq  \int_{B_{4L}(0)}|\Tilde{f}|+|D\Tilde{f}| \di x + \sum_{k=1}^j \int_{B_{4L}(0)}|\check{f}_k|+|D\check{f}_k| \di x \\
         & \geq \frac{(a-\e)\mathcal{L}^1(K)}{L} + \sum_{k=1}^j \frac{\mathcal{L}^1(J_k)}{L} (1-a-2\e) \geq \frac{(1-3\e)\mathcal{L}^1(K)}{L}.
    \end{align*}
    As $\e$ was arbitrary, the lemma follows. 
\end{proof}
Roughly speaking, Lemma \ref{2dcapacity_of_line_segments} establishes a uniform $1$-capacity bound for this particular family of sets. This is summarised in the next Corollary.
\smallskip

\begin{corollary}
    Given $L>1$, $\alpha\in [0,1]$, then there exists a constant $C=C(L, \alpha)>0$ such that the following holds: For any Borel set $A \subset [0,1]$ with $\mathcal{L}^1(A) \geq \alpha$ and any $(\frac{1}{L}, L)$-bi-Lipschitz map $\gamma:[0,1]\to \R^2$, it holds that $\Capa_1(\gamma(A)) \geq C$.
\end{corollary}
\begin{proof}
    Fix $\e>0$. Let $\Tilde{f} \in W^{1,1}(\R^2)$ such that $\gamma(A) \subset \{\Tilde{f}=1\}^\circ$ and such that 
    \begin{align*}
        \Capa_1(\gamma(A)) \geq \int_{\R^2} |D\Tilde{f}| \, \di x -\e.
    \end{align*}
    Define the open set $U:= \{\Tilde{f}=1\}^\circ \subset \R^2$. Then there exists an open set $O \subset [0,1]$ such that $A \subset O$ and $\gamma(O) \subset U$. It follows that $\mathcal{L}^1(O) \geq \alpha$. Now we can find a finite union of pairwise disjoint compact intervals $I_1, \ldots, I_n \subset [0,1]$ such that $K:= \bigcup_{i=1}^n I_i \subset O$ and 
    \begin{align*}
        \mathcal{L}^1\Big(\bigcup_{i=1}^n I_i\Big)=\sum_{i=1}^n \mathcal{L}^1(I_i) \geq \frac{\alpha}{2}.
    \end{align*}
    Via regularisation, we can find a function ${f}\in C_c^\infty(\R^2)$ that satisfies ${f} \equiv1$ on $\gamma(K)$ and 
    \begin{align*}
        \int |D\Tilde{f}-Df|\, \di x \leq \e. 
    \end{align*}
    Lemma \ref{2dcapacity_of_line_segments} now yields that 
    \begin{align}\label{cor_sobolev_estimate_capa}
        \int_{B_{4L}(\gamma_0)} |f|+|Df| \, \di x \geq \frac{\alpha}{2L}.
    \end{align}
    Now, we can distinguish two cases:\\
    \textbf{Case 1.}
         \begin{align*}
        \int_{B_{4L}(\gamma_0)} |Df| \, \di x \geq \frac{\alpha}{4L}.
    \end{align*}
    \textbf{Case 2.}
    \begin{align*}
        \int_{B_{4L}(\gamma_0)} |Df| \, \di x < \frac{\alpha}{4L}.
    \end{align*}
    Then, \eqref{cor_sobolev_estimate_capa} implies that 
    \begin{align*}
        \int_{B_{4L}(\gamma_0)} |f| \, \di x \geq \frac{\alpha}{4L}.
    \end{align*}
    Using the Cauchy-Schwartz-inequality, this implies
    \begin{align*}
        \norm{f}_{L^2(B_{4L}(\gamma_0))} \geq \frac{\norm{f}_{L^1(B_{4L}(\gamma_0))}}{\sqrt{|B_{4L}(0)|}} \geq \frac{\alpha}{30L^2}.
    \end{align*}
    By the Sobolev embedding theorem, we get that there exists a constant $c>0$ such that 
    \begin{align*}
        \int_{\R^2} |Df| \, \di x \geq c \Big(\int |f|^2\, \di x\Big)^{\frac{1}{2}} \geq \frac{c\alpha}{30L^2}.
    \end{align*}
    As either Case 1 or Case 2 must be true, we get that 
    \begin{align*}
        \Capa_1(\gamma(A)) \geq \min\Big(\frac{c\alpha}{30L^2}, \frac{\alpha}{4L} \Big)-2\e.
    \end{align*}
    The corollary now follows as $\e$ can be chosen arbitrarily small.
\end{proof}

\section{Application to manifolds with non-continuous Riemannian metrics}\label{sec:application}
\subsection{First order calculus on metric measure spaces}\label{sec:2}
In this section we will recall notions from first order calculus on metric measure spaces.
We will briefly summarise some tools to define a generalised notion of modulus of gradients and Sobolev functions in metric measure spaces that was introduced by Cheeger in \cite{cheeger1999differentiability} and further analysed by Ambrosio, Gigli and Savaré  \cite{ambrosio2014inventio}. The structure of this section has been inspired by \cite{mondino2025equivalence}.

Let $(X, \sfd)$ be a complete and separable metric space. 
We endow $(X,\sfd)$ with a non-negative $\sigma$-finite Borel  measure $\mathfrak{m}$, obtaining the metric measure space $(X, \sfd, \mathfrak{m})$.
Throughout the rest of this work, we assume that there exists a bounded Borel Lipschitz map $V: X \to [0, \infty)$ such that  (cf. \cite[Ch.\,4]{ambrosio2014inventio})
\begin{equation}\label{condition_on_measure_bounded_compression}
\begin{split}
    V\ \mathrm{is\ bounded\ on\ each\ compact\ set\ } K \subset X \ \mathrm{and}  \\
    \int_X e^{-V^2} \, \di\mathfrak{m} \leq 1. 
    \end{split}
\end{equation}
\begin{definition}
    Let $\gamma \in C([0,1],X)$ be a curve. Then we say that $\gamma$ is absolutely continuous if there exists a function $g \in L^1([0,1], [0, \infty))$ such that for all $0 \leq s<t \leq 1$, it holds
    \begin{align*}
        \sfd(\gamma_s, \gamma_t) \leq \int_s^t g(r) \, \di r.
    \end{align*}
    There exists a minimal such function $g$, which we call the metric speed of $\gamma$ and denote it by $|\dot{\gamma}_t|$. 
\end{definition}
The metric speed of an absolutely continuous curve $\gamma:[0,1] \to X$ is given by  $|\dot{\gamma}_t|:= \lim_{h \to 0}\frac{\sfd(\gamma_{t+h}, \gamma_t)}{|h|}$ {which exists for almost every $t \in [0,1]$} \cite{ambrosio1990metric}, (see also \cite{kirchheim1994rectifiable}).

We next recall the notion of test plan and weak upper gradient. We will use the conventions of  \cite{gigli2018nonsmooth}, after \cite{ambrosio2014inventio}. For any metric space $(Y, \sfd_Y)$, we denote by $\mathcal{P}(Y)$ the set of Borel probability measures on $Y$. Given two topological spaces $X_1, X_2$, a Borel-measurable map $T:X_1 \to X_2$, and a Borel measure $\mu$ on $X_1$, the pushforward measure $T_\#\mu$ on $X_2$ is the Borel measure defined by $T_\#\mu(B) = \mu (T^{-1}(B))$ for any Borel set $B \subset X_2$.
\smallskip

\begin{definition}\label{test_plan}
    Let $\boldsymbol{\pi} \in \mathcal{P}(C([0, 1], X))$. We say that $\boldsymbol{\pi}$ is a \emph{test plan} if there exists a constant $C(\boldsymbol{\pi})>0$ such that 
    \begin{align*}
        (e_t)_\#\boldsymbol{\pi} \leq C(\boldsymbol{\pi})\mathfrak{m}\ \mathrm{for \ all \ } t \in [0,1]
    \end{align*}
    and 
    \begin{align*}
        \int \int_0^1 |\dot{\gamma}_t|^2\,\di t\di\boldsymbol{\pi}(\gamma) < \infty. 
    \end{align*}
    We use the convention that if $\gamma$ is not absolutely continuous, then $\int_0^1 |\dot{\gamma}_t|^2\,\di t =\infty$.
\end{definition} 
In Question \ref{q2}, ``almost every absolutely continuous curve" refers to a set of curves $\Gamma \subset C([0,1], M)$ such that for every test plan $\boldsymbol{\pi}$ as defined above, $\boldsymbol{\pi}(\Gamma)=1$. 
\smallskip

\begin{definition}
   Given { $f: X \to \R$ } a $\mathfrak{m}$-measurable function, then a $\mathfrak{m}$-measurable function $G: X \to [0, \infty]$ is called a weak upper gradient of $f$, if 
   \begin{align}\label{weakgradient}
       \int|f(\gamma_1)-f(\gamma_0)|\,\di\boldsymbol{\pi}(\gamma)\,  \leq \int\int_0^1G(\gamma_t)|\dot{\gamma}_t|\, \di t\di\boldsymbol{\pi}(\gamma) < \infty, \quad \mathrm{for \ all \ test \ plans } \ \boldsymbol{\pi}. 
   \end{align}
\end{definition}
The discussion in \cite{ambrosio2014inventio}, Prop. 5.9 and Def. 5.11 shows the existence of a weak upper gradient $|Df|_w$ such that $|Df|_w \leq G$ {$\mathfrak{m}$-a.e.} for all other weak upper gradients $G$. We will call it the \textit{minimal weak upper gradient} of $f$.
\begin{definition}
    The Cheeger energy is defined in the class of $\mathfrak{m}$-measurable functions by
    \begin{align*}
        \Ch(f):= \left\{ \begin{array}{ll}
           \frac{1}{2}\int_X |Df|^2_w\, \di\mathfrak{m}  & \mathrm{if}\ f\ \mathrm{has\ a\ weak\ upper\ gradient\ in }\ L^2(X, \mathfrak{m}),  \\
           \infty  & \mathrm{otherwise},
        \end{array}\right.
    \end{align*}
    with proper domain $D(\Ch)= \{f:X \to [0, \infty], \mathfrak{m}\mathrm{-measurable}, \ \Ch(f) < \infty\}$.
\end{definition}
\begin{definition}
    A metric measure space is called \emph{infinitesimally Hilbertian} if  $\Ch$ is a quadratic form.
\end{definition}

\subsection{The distance on manifolds with non-continuous Riemannian metrics}\label{sec:3}
In this subsection we discuss the definition and basic properties of the Riemannian distance on manifolds with non-continuous Riemannian metrics. This summary is inspired by \cite[Section 3]{ryborz2026infinitesimal}.

From now on we will consider a smooth $d$-dimensional manifold $M$ with a Riemannian metric $g$ such that $g, g^{-1} \in L^\infty_{\rm loc}$.
Let $U \subset M$ be open such that $U$ lies in one coordinate patch and let $\psi: U \to \R^d$ be a local trivialisation. Then, $g, g^{-1} \in L^\infty_{\rm loc}(M)$ means that for every such $U$, the components of $\psi_*g, \psi_*g^{-1}$ are equivalence classes of locally $\mathcal{L}^d$-measurable functions up to $\mathcal{L}^d$-almost everywhere equality in $\psi(U)$ and that for every compact set $K \subset \psi(U)$, there exist constants $\lambda, \Lambda >0$ such that $\lambda \mathrm{Id} \leq \psi_*g \leq \Lambda \mathrm{Id}$ $\mathcal{L}^d$-almost everywhere in $K$. Note that under these assumptions the measures $\mathcal{L}^d$ and $\vol_g$ are locally equivalent and therefore have the same null sets.

Then $(M, g)$ falls into the class of LIP-Riemannian manifolds (see for example \cite{teleman1983index, de1991integral}). They have been studied by Norris \cite{norris1997heat}, and De Cecco and Palmieri  \cite{de1991integral}, as well as  Saloff-Coste  \cite{saloff1992uniformly} and Sturm  \cite{sturm1995analysis} with emphasis on potential theory for uniformly elliptic operators. 
Note that LIP-manifolds shall not be confused with metric tensors whose coefficients are locally Lipschitz continuous.

Recall that for smooth manifolds equipped with a continuous Riemannian metric $g$, the Riemannian distance $\sfd_g$ is defined as follows
\begin{align}\label{def:dg}
    \sfd_g(x, y) = \inf\Big\{ \int_{0}^1 |\dot{\gamma_t}|_g\, \di t: \, \gamma\ \mathrm{piecewise \ }C^\infty, \gamma_0=x, \gamma_1 = y \Big\},
\end{align}
for any $x, y \in M$.

If $g$ is merely measurable, it might not be well-defined everywhere, hence the definition of the Riemannian distance $\sfd_g$ as in \eqref{def:dg}, cannot be used for non-continuous Riemannian metrics. 
However, Norris \cite{norris1997heat} and De Cecco-Palmieri \cite{de1991integral} gave a definition of a distance $\sfd_g$ on $M$ arising from $g$ that avoids this difficulty. 

In the case of a Riemannian manifold with $g, g^{-1}\in L^\infty$, Norris defines the distance $\sfd$ via
 \begin{align*}
     \sfd(x,y) := \sup \left\{ \left[ w(x)-w(y) \right] ,\ w \mathrm{\ is\ Lipschitz\ continuous\ and\ } |\nabla_g w|_g\leq 1\ \mathrm{a.e. } \right\}.
 \end{align*}
 Moreover, for $\rho\in C_c^\infty(B_1^{\R^d}(0), [0,\infty))$, with $\int_{\R^d} \rho\, \di x =1$, define $\rho_\e := \frac{1}{\e}\rho(\frac{x}{\e})$ and $g_\e := \rho_\e * g$ (for details on the convolution on manifolds see \cite{grosser2013geometric, graf2020singularity}). The metrics $g_\e$ are continuous and hence the distance $\sfd_{g_\e}$ is given as in \eqref{def:dg}. Norris then proves that for $\e \to 0$, $\sfd_{g_\e}$ converges to a distance $\sfd_0$ that is locally equivalent to the underlying Euclidean metric induced from a coordinate chart; in particular,  \cite[Theorem 3.6]{norris1997heat} states that $\sfd_0 = \sfd =: \sfd_g$. 

 From the properties of $g$, it follows that all points outside a null set $N_g \subset M$ are Lebesgue points of $g, g^{-1}$. For any null set $N \subset M$, define the set 
 \begin{align}
     \mathrm{Lip}_N(x, y, M):= \{\gamma \in C^{0,1}_{\rm{loc}}([0,1], M): \gamma(0)=x, \gamma(1)=y, \mathcal{L}^1(\{t \in [0,1]: \gamma(t) \in N\}) =0\}.
 \end{align}
 If $\gamma \in C([0,1], M)$ satisfies $\mathcal{L}^1(\{t \in [0,1]: \gamma(t) \in N\}) =0$, we say that $\gamma$ is transversal to $N$ and write $\gamma \perp N$.

Similarly to \cite{ryborz2026infinitesimal}, we will from now on make the following assumption:

\smallskip
\begin{assumption}\label{assumption_on_mfd}
The manifold $(M,g)$ satisfies one of the following:
    \begin{itemize}
    \item[(i)] The manifold $M$ has finite $\vol_g$-volume. 
    \item[(ii)] The manifold $M$ has infinite $\vol_g$-volume and our chosen atlas satisfies that there exists a constant $h \geq 0$ such that all coordinate charts in our atlas satisfy that the smallest eigenvalue of $g$ is uniformly bounded below $h$ and $M$ is complete or the interior of a complete manifold (where the completeness is with respect to the distance induced by the charts as in \cite[(1.9)]{de1991integral}).
\end{itemize}
\end{assumption}

  De Cecco and Palmieri proved in \cite[Theorem (2.18) and Theorem (6.1)]{de1991integral}  that under Assumption \ref{assumption_on_mfd}, it holds $\sfd_g= \boldsymbol{\delta}$, where the distance $\boldsymbol{\delta}$ is given by 
 \begin{align}\label{null_set_def_of_metric}
    \boldsymbol{\delta}(x, y)= \sup_{N\subset M, \vol_g(N)=0} \boldsymbol{\delta}_N(x, y),
 \end{align}
where
 \begin{align}
     \boldsymbol{\delta}_N :=\inf\{L_g(\gamma), \gamma \in \Lip_N(x, y, M)\},
 \end{align}
 and $L_g(\gamma) = \int_0^1 |\dot{\gamma}_t|_g \di t$.
 In \cite[Theorem 4.4]{de1988distanza}, it is shown that for each null set $N_0 \subset M$, $\boldsymbol{\delta} = \boldsymbol{\delta}_0:= \sup_{N, \vol_g(N)=0} \boldsymbol{\delta}_{N \cup N_0}$. 

 Note that the volume measure $\vol_g$ is locally given as $\sqrt{\det g}\di \mathcal{L}^d$, which now allows us to consider the metric measure space $(M, \sfd_g, \vol_g)$. In the following we recall some basic properties of our setting:
\smallskip
  
 \begin{proposition}[Proposition 3.2 in \cite{ryborz2026infinitesimal}]
     Let $\gamma:[0, 1] \to M$ be a $\sfd_g$-absolutely continuous curve such that $\gamma \perp N_g$. Then, for almost every $t \in [0,1]$, it holds
     \begin{align*}
          |\dot{\gamma}_t| \leq |\dot{\gamma}_t|_g.
     \end{align*}
 \end{proposition}
\smallskip
 
 \begin{proposition}[Proposition 3.3 in  \cite{ryborz2026infinitesimal}]\label{metric_speed_depends_on_direction}
 Let $M$ be a smooth manifold and $g$ a Riemannian metric such that $g, g^{-1} \in L^\infty_{\rm loc}$. 
     Let $\gamma:[0,1] \to M$ be a Lipschitz continuous curve and $t\in [0,1]$ such that the metric speed $|\dot{\gamma}_t|$ exists at $t$ and $\gamma$ is differentiable at $t$. Let $c:[0,1] \to M$ be another curve such that $c_t = \gamma_t$, $c$ is differentiable at $t$ and $\dot{c}_t = \dot{\gamma}_t$. Then the metric speed of $|\dot{c}_t|$ of $c$ at $t$ exists and $|\dot{c}_t| = |\dot{\gamma}_t|$. 
 \end{proposition}

\smallskip
 
 \begin{proposition}[Proposition 3.4 in \cite{ryborz2026infinitesimal}]\label{metric_speed_continuous_from_direction}
     Let $\gamma^1, \gamma^2 :[0,1] \to M$ be two Lipschitz continuous curves that are contained in a compact set $K$ inside one coordinate patch $V$ and $t \in [0,1]$ such that $\gamma^1_t = \gamma^2_t$ and such that the metric speeds $|\dot{\gamma}^1_t|$ and $|\dot{\gamma}^2_t|$ as well as their derivatives $\dot{\gamma}^1_t,\dot{\gamma}^2_t\in T_{\gamma^1_t}M$ exist. Then there exists a constant $C =C(K)$ such that 
     \begin{align}
         ||\dot{\gamma}^1_t|-|\dot{\gamma}^2_t|| \leq C|\dot{\gamma}^1_t-\dot{\gamma}^2_t|_g.
     \end{align}
 \end{proposition}

 \begin{lemma}[Lemma 3.5 in \cite{ryborz2026infinitesimal}]\label{testplans_on_transversal}
     Let $\boldsymbol{\pi}\in \mathcal{P}(C([0,1], M))$ be a test plan and $N\subset M$ a $\vol_g$-null set. Then, it holds that 
     \begin{align}
         \boldsymbol{\pi}\Big(\{\gamma \in C([0,1], M): \gamma \perp N \}\Big)=1.
     \end{align}
 \end{lemma}
 
\subsection{Infinitesimal Hilbertianity for two dimensional manifolds with Sobolev regular Riemannian metrics}\label{sec:4}

In this section we will use the capacity theory from the previous section to study the case $d=2$. Some of the preparatory results remain true in higher dimensions than $2$, so they have been formulated in higher generality.
\smallskip

We assume to be in the setting of a $d$-dimensional manifold with a Riemannian metric $g$ such that $g, g^{-1}\in L^\infty_{\rm loc}$. We work with an atlas that satisfies Assumption \ref{assumption_on_mfd}. Denote by $\sfd_g$ the induced Riemannian distance on $M$. The next two results hold for $d \geq 2$. 
\begin{lemma}\label{metric_speed_exists_ae}
Let $d \geq 2$. 
    Let $U \subset M$ such that $U$ lies in one coordinate patch and let $v \in \R^d$. Then for almost every $x \in U$ the metric speed of the curve $\gamma^x: t \mapsto x+tv$ exists at $t=0$.
\end{lemma}
Similar observations were made by de Cecco and Palmieri \cite{de1995lip}, when they analysed what they called the ``derivative" of the distance function. 
\begin{proof}
We may assume that $U \subset \R^d$.
    Denote $\vartheta:= |v|_{euc}$.
    Fix an $y\in U$ and find $r>0$ such that $B^{euc}_{4r\vartheta}(y) \subset U$.
    For $x \in B_{r\vartheta}(y)$, reparametrise the curve $\gamma^x:[0,1]\to U, t \mapsto x+r(\frac{1}{2}-t)v$. 
    Define the test plan $\boldsymbol{\pi}$ via
    \begin{align*}
        \di \boldsymbol{\pi}(\gamma)= \left\{ \begin{array}{ll}
            \frac{1}{\mathcal{L}^{d+1}(B_{r\vartheta}(y))}\di \mathcal{L}^{d}(x)  & \mathrm{if}\ \gamma=\gamma^x \ \mathrm{for\ some\ } x \in B_{r\vartheta}(y), \\
            0 & \mathrm{otherwise.}
        \end{array} \right.
    \end{align*}
    Denote by $N$ the set of points $x$ for which the metric speed of $\gamma^x$ at $t=0$ does not exist. Now, as $\gamma^x$ is absolutely continuous, we have that $\gamma^x \perp N$ for every $x \in B_{r\vartheta}(y)$. Note that there exists a $c>0$ such that for each $t\in [\frac{1}{4},\frac{3}{4}]$, it holds 
    \begin{align*}
        (e_t)_{\#}\boldsymbol{\pi} \geq c \mathcal{L}^d\LL {B_{r\vartheta/8}(y)}.
    \end{align*}
    Now Fubini's theorem yields that
    \begin{align*}
        \frac{c}{2} \mathcal{L}^d (N \cap B_{r\vartheta/8}(y)) \leq  \int_0^1 \int_M \mathbbm{1}_{N}(y)\,\di(e_t)_\#\boldsymbol{\pi}(y)\di t =  \int_{C([0,1], M)} \int_0^1 \mathbbm{1}_{N}(\gamma_t)\,\di t\di\boldsymbol{\pi}(\gamma)=0.
    \end{align*}
    This can be done for any $y \in U$ and hence proves the lemma. 
\end{proof}

\begin{lemma}\label{lemma_directional_metric_speed_measurable}
Let $d \geq 2$.
    Let $U \subset M$ such that $U$ lies in one coordinate patch and let $v \in \R^d$. Define the curve $\gamma^x: t \mapsto x+tv$. Then the $\mathcal{L}^d$-a.e. defined map $x \mapsto |\dot{\gamma}^x_0|$ is Lebesgue measurable.
\end{lemma}
\begin{proof}
We may assume that $U \subset \R^d$.
    Denote $\gamma^x:[-\delta,\delta] \to U, t \mapsto x+tv$. By Lemma \ref{metric_speed_exists_ae}, we get that the metric speed of the curve $\gamma^x$ exists at $0$ for almost every $x$, and it is given by 
    \begin{align*}
        |\dot{\gamma}^x_0| = \lim_{h \to 0} \frac{\sfd_g(x, x+hv)}{|h|}.
    \end{align*}
    Now for $n\in \N$ define the function 
    \begin{align*}
        s(x, n):= n\sfd_g\Big(x, x+\cfrac{v}{n}\Big).  
    \end{align*}
    As $\sfd_g$ is continuous on $U^2$, we get that for all $n$, it holds that $s(\cdot, n)$ is a measurable function. Hence, $\limsup_{n \to \infty} s(\cdot, n)$ is measurable. Noting that $\limsup_{n \to \infty} s(x, n) = |\dot{\gamma}^x_0|$ for almost every $x$, we get that $x \mapsto |\dot{\gamma}^x_0|$ is Lebesgue measurable.
\end{proof}

We have now collected all ingredients to answer Question \ref{q1} in two dimensions.
\smallskip

\begin{proposition}\label{2d_metric_speed_sobolev}
    Let $d=2$ and $p \geq 1$. Let $M$ be a $d$-dimensional manifold and $g$ a Riemannian metric on $M$ such that $g, g^{-1} \in L^\infty_{\rm loc} \cap W^{1,p}_{\rm loc}$ and suppose that Assumption \ref{assumption_on_mfd} is satisfied. Then the following holds for almost every $x \in M$: Given a curve $\gamma\colon[0,1] \to M$ with $\gamma_t =x$ for some $t\in (0,1)$ and such that both the metric speed $|\dot{\gamma}_t|$ of $\gamma$ at time $t$ and the derivative of $\dot{\gamma}_t \in T_xM$ at time $t$ exist, then $|\dot{\gamma}_t| = |\dot{\gamma}_t|_{g(x)}$.
\end{proposition}
\begin{proof}
Proving the proposition for $p=1$ is sufficient as $W^{1,p}_{\rm loc} \subset W^{1,1}_{\rm loc}$ for all $p \geq 1$. Hence, we may assume that $p =1$.
We argue by contradiction and suppose the statement is false. \\
\textbf{Step 1.} Reduction to one coordinate patch and one direction. \\
    Let $(U_i, \psi_i)_{i \in \N}$ be a countable atlas of relatively compact coordinate patches of $M$. Let furthermore $(v_j)_{j\in \N}$ be a dense set in $\R^{2}$. Fixing one $i \in \N$, we may assume that $U_i \subset \R^2$. 
    By Lemma \ref{metric_speed_exists_ae}, we have that for almost every $x \in U_i$, the metric speed of the curve $t \mapsto x+tv_j$ exists at $t=0$ for all $j \in \N$. We denote the complement of that set by $N_i$ and note that $N_i$ is a $\vol_g$ (and hence $\mathcal{L}^2$)-null set.
    Fix an $x \in U_i \setminus N_i$ such that there exists a curve $\gamma:[0,1] \to M$ with $\gamma_t =x$ for some $t$ and such that the metric speed $|\dot{\gamma}_t|$ of $\gamma$ at time $t$ and the derivative of $\dot{\gamma}_t \in T_xM$ at time $t$ exist, but $|\dot{\gamma}_t| \neq |\dot{\gamma}_t|_{g(x)}$. Denote $v:= \dot{\gamma}_t \in T_xM$. Now by Proposition \ref{metric_speed_continuous_from_direction} and Proposition \ref{metric_speed_depends_on_direction}, we get that there exists a $j \in \N$ such that $\gamma^{x,j}: t \mapsto x+tv_j$ satisfies $|\dot{\gamma}^{x,j}_0| \neq |v_j|_{g(x)}$.
    Define the set 
    \begin{align*}
        S_{i,j, n}:= \Big\{x \in U_i \setminus N_i: \ ||\dot{\gamma}^{x,j}_0| - |v_j|_{g(x)}| \geq \frac{1}{n} \Big\}. 
    \end{align*}
    Note that by Lemma \ref{lemma_directional_metric_speed_measurable}, $S_{i,j,n}$ is measurable.
    Now if the statement of the proposition was not true, that translates into 
    \begin{align*}
       \vol_g\Big( \bigcup_{i,j,n \in \N} S_{i,j,n}\Big) >0.
    \end{align*}
    It then follows that there must exist some $i, j, n \in \N$ such that $\vol_g (S_{i,j,n})>0$. Using that each $U_i$ is relatively compact, and hence $\mathcal{L}^2$ and $\vol_g$ are equivalent on $U_i$, it follows that: \\

    There exists an open, relatively compact set $U \subset \R^2$, a vector $v \in \R^2$ and $\zeta > 0$, $\delta >0$, and a Lebesgue measurable set $S \subset U$ such that
    \begin{align}
        \mathcal{L}^2(S)>\delta,
    \end{align}
     and such that for each $x \in S$, the metric speed $|\dot{\gamma}^{x,v}_0|$ of the curve $\gamma^{x,v}: t \mapsto x+tv$ at $0$ exists and satisfies
    \begin{align}\label{set_zeta}
        |\dot{\gamma}^{x,v}_0| \leq |v|_{g(x)}-\zeta.
    \end{align}
    We may furthermore assume that $|v|_{euc} =1$.
    This concludes the first step. \\

    \textbf{Step 2.} The existence of faulty lines and their properties. \\
     We now define the $\Capa_1$/$\mathcal{L}^d$-null set $N$ as the set of non-Lebesgue points of $g$. Now for each $x \in S \setminus N$, there must exist a $r_x>0$ such that for $r \in (0,r_x)$ it holds $\sfd_g(x, x+rv) < |v|_{g(x)}-\frac{3\zeta}{4}$.
     Recall that $\sfd_g = \lim_{\varsigma \to 0} \sfd_{g_\varsigma}$, where $g_\varsigma := \rho_\varsigma * g$. Now, it follows that for each $x \in S \setminus N$ and $r \in (0, r_x)$, there exists a $\varsigma_{x,r}>0$ such that for all $\varsigma \in (0, \varsigma_{x,r})$, it holds
    \begin{align}\label{distance_assumption_step2_2d}
        |g(x)-g_\varsigma(x)|_{op} \leq \frac{\zeta}{4} \ \mathrm{and}\ \sfd_{g_\varsigma}(x, x+rv) < r\Big(|v|_{g(x)}- \frac{\zeta}{2}\Big).
    \end{align}
    As $g_\varsigma$ is smooth, the Arzela-Ascoli theorem yields that there exists a length minimising Lipschitz continuous curve $\gamma^{r, \varsigma}:[0,r] \to U$ such that
    \begin{align}\label{shorter_curves2d}
        \gamma^{r, \varsigma}_0 =x, \ \gamma^{r, \varsigma}_r = x+rv, \ \int_0^r |\dot{\gamma}^{r, \varsigma}_t|_{g_\varsigma(\gamma^{r, \varsigma}_t)}\,  \di t = \sfd_{g_\varsigma}(x, x+rv) < r\Big(|v|_{g(x)}- \frac{\zeta}{2}\Big).
    \end{align}
    Let $0<\lambda'<1< \Lambda'$ such that on $U$ it holds $\lambda' \mathrm{Id}_2 \leq g \leq \Lambda' \mathrm{Id}_2$ and such that $\sfd_g$ and $\sfd_{euc}$ are $(\lambda', \Lambda')$-equivalent in $U$. As $U$ is relatively compact, we can shrink $\varsigma_{x, r}>0$ and find $\lambda, \Lambda>0$ with $0<\lambda < \lambda' < 2\lambda <1<\Lambda<2\Lambda'$ such that for all $\varsigma \in (0, \varsigma_{x, r})$ it holds that $\lambda \mathrm{Id}_2 \leq g_\varsigma \leq \Lambda \mathrm{Id}_2$ and such that $\sfd_{g_\varsigma}$ and $\sfd_{euc}$ are $(\lambda, \Lambda)$-equivalent in $U$.
    
    Fix a $\varsigma \in (0, \varsigma_{x, r})$ and write from now on $\gamma^r:= \gamma^{r, \varsigma}$.
    It follows that 
    \begin{align}
         \int_0^r |\dot{\gamma}^r_t|_{euc} \di t \leq \Lambda \int_0^r |\dot{\gamma}^r_t|_{g_\varsigma(\gamma^r_t)}\,  \di t \leq 2 \Lambda^2r,
    \end{align}
    hence we may reparametrise $\gamma^r$ to be of constant Euclidean speed and get that the Lipschitz constant is bounded from above by $2\Lambda^2$. As $\sfd_{g_\varsigma}(x, x+rv) \geq \lambda r$, we get that the Euclidean speed is bounded from below by $\frac{\lambda}{2\Lambda}$, as otherwise one would have
    \begin{align*}
        \lambda r \leq \int_0^r |\dot{\gamma}^r_t|_{g_\varsigma(\gamma^r_t)} \, \di t \leq \Lambda \int_0^r |\dot{\gamma}^r_t|_{euc} \, \di t \leq \frac{\lambda r}{2}.
    \end{align*} 
    It follows that for each $s, s' \in [0,r]$, $s<s'$ it holds
    \begin{align}\label{minimal_length_of_curve_segment_2d}
        \int_s^{s'} |\dot{\gamma}^r_t|_{g_\varsigma(\gamma^r_t)}\, \di t \geq \lambda \int_s^{s'} |\dot{\gamma}^r_t|_{euc}\, \di t \geq \frac{\lambda^2}{2\Lambda} |s-s'|.
    \end{align} 
   \textbf{Claim.} For each $s, s' \in [0,r]$, $s<s'$ it holds 
    \begin{align}
        \frac{\lambda^3}{4\Lambda^2} |s-s'| \leq \sfd_{euc}(\gamma_s, \gamma_{s'}).
    \end{align}
    \textit{Proof of the claim.}
    Suppose this was not the case for some $s, s' \in [0,r]$. Then we could replace the segment between $\gamma_s$ and $\gamma_{s'}$ by a straight line segment between $\gamma_s$ and $\gamma_{s'}$. That would give a curve $\Tilde{\gamma}:[s, s'] \to U$ such that $\Tilde{\gamma}_s = \gamma_s$ and $\Tilde{\gamma}_{s'} = \gamma_{s'}$ with 
    \begin{align*}
        \int_s^{s'} |\dot{\Tilde{\gamma}}_t|_{euc} \di t \leq \frac{\lambda^3}{4\Lambda^2} |s-s'|
    \end{align*}
    which using \eqref{minimal_length_of_curve_segment_2d} yields
    \begin{align}
        \int_s^{s'} |\dot{\Tilde{\gamma}}_t|_{g_\varsigma(\Tilde{\gamma}_t)}\, \di t \leq \frac{\lambda^3}{4\Lambda} |s-s'| < \int_s^{s'} |\dot{\gamma}^r_t|_{g_\varsigma(\gamma^r_t)}\, \di t,
    \end{align}
    so replacing $\gamma^r$ with $\Tilde{\gamma}$ on $[s, s']$ yields a shorter curve.
    This contradicts the length optimality of $\gamma^r$ with respect to $g_\varsigma$ and hence proves the claim.  
    
    This yields that we can assume that there exists a constant $C>1$ that only depends on $U$ and $g$ such that for each $0 \leq s< s' \leq 1$, it holds
    \begin{align}\label{bi-Lipschitz_curves2d}
        \frac{1}{C}|s-s'|\leq \sfd_{euc}(\gamma^r_s, \gamma^r_{s'})\leq C |s-s'|.
    \end{align}
    Observe that by the properties of inner products on $\R^2$, it holds
    \begin{align*}
        |rv|_{g(x)} =\Big|\int_0^r \dot{\gamma}^r_t\, \di t\Big|_{g(x)} \leq \int_0^r |\dot{\gamma}^r_t|_{g_\varsigma(x)}\, \di t .
    \end{align*}
    Note that in the last integral, we evaluate every vector with respect to $g_\varsigma(x) = g_\varsigma(\gamma^r_0)$ instead of $g_\varsigma(\gamma^r_t)$. 
    By \eqref{distance_assumption_step2_2d}, it follows that 
    \begin{align*}
        \int_0^r |\dot{\gamma}^r_t|_{g(x)} \di t \geq |rv|_{g(x)} \geq \sfd_{g_\varsigma}(x, x+rv) + \frac{\zeta r}{2} = \int_0^r |\dot{\gamma}^r_t|_{g_\varsigma(\gamma^r_t)} \di t + \frac{\zeta r}{2}.
    \end{align*}
    Hence, 
    \begin{align}\label{integral_speed_comparison2d}
        \int_0^r |\dot{\gamma}^r_t|_{g(x)} -|\dot{\gamma}^r_t|_{g_\varsigma(\gamma^r_t)} \di t \geq \frac{\zeta r}{2}.
    \end{align}
    Note that \eqref{bi-Lipschitz_curves2d} implies that for all $t \in [0,r]$, it holds $|\dot{\gamma}^r_t|_{g(x)} -|\dot{\gamma}^r_t|_{g_\varsigma(\gamma^r_t)} \leq C(U, g)$ and hence, combining this with \eqref{integral_speed_comparison2d} yields
    \begin{align}
        \mathcal{L}^1\Big(\Big\{t: \ |\dot{\gamma}^r_t|_{g(x)} -|\dot{\gamma}^r_t|_{g_\varsigma(\gamma^r_t)} \geq \frac{\zeta }{4} \Big\}\Big) \geq \frac{\zeta r}{4C(U, g)}.
    \end{align}
    Using \eqref{bi-Lipschitz_curves2d} once more, gives that there exists a constant $C=C(g, U)>1$ and a set $A^{x, r, \varsigma} \subset [0,r]$ such that 
    \begin{align}\label{discontinuity_of_g_along_line2d}
        &\mathcal{L}^1(A^{x, r, \varsigma}) \geq \frac{r\zeta}{4C},  \nonumber \\
        &\forall t \in A:\ \norm{g(x)-g_\varsigma(\gamma^{r, \varsigma}_t)}_{euc} \geq \frac{\zeta}{4C},
    \end{align}
    where in the last inequality we evaluate the Euclidean matrix norm. This finishes the second step.\\

    \textbf{Step 3.} The faulty lines have big $1$-capacity. \\
    Let us fix some constants for this step. We fix $C=C(U, g)>1$ for \eqref{bi-Lipschitz_curves2d} and \eqref{discontinuity_of_g_along_line2d} to hold. We fix $\zeta \in (0,1)$ for \eqref{set_zeta} (and hence \eqref{discontinuity_of_g_along_line2d}) to be true. Finally, $C_2$ shall denote the constant from the Besicovitch covering theorem in $\R^2$ and $c_{1,2}$ denotes the constant such that for each $\phi \in W^{1,1}(\R^2)$, it holds $c_{1,2}\norm{\phi}_{L^2} \leq \norm{D\phi}_{L^1}$.  
    Define
    \begin{align}\label{set_eps}
        \epsilon:= \frac{\zeta}{1600\pi C_2C^5}\cdot \frac{\delta}{2} >0,
    \end{align}
    By Theorem \ref{quasicontinuity} and Theorem \ref{basic_properties_capa} (v), we can find an open set $\Omega \subset U$ such that 
    \begin{align}\label{Capacity_omega}
        \Capa_1(\Omega) \leq \min\Big(1, \frac{c_{1,2}}{\sqrt{\mathcal{L}^2(U)}}\Big) \cdot \frac{\epsilon}{80}, \ \mathcal{L}^2(\Omega) \leq \frac{\delta}{2}
    \end{align}
    such that $g$ is continuous when restricted to $U \setminus \Omega$ and by the discussion after Theorem \ref{quasicontinuity} we may assume that (a subsequence of) $g_\varsigma$ converges uniformly on $U \setminus \Omega$. Note that 
    \begin{align}\label{remaining_S}
        \mathcal{L}^2(S \setminus (\Omega\cup N)) \geq \frac{\delta}{2}.
    \end{align}
    Let $f \in W^{1,1}(\R^2, [0,1])$ such that $f \equiv 1$ on $\Omega$ and 
    \begin{align}\label{function_with_low_capacity}
        \int_{\R^2} |Df|\, \di x \leq 2 \Capa_1(\Omega).
    \end{align}
    Now, for every $x \in S \setminus (\Omega\cup N)$, there exists $\e_x >0$ such that for all $y \in B_{\e_x}(x) \setminus \Omega$, it holds 
    \begin{align}\label{g_continuous_away_omega2d}
        \norm{g(x)-g(y)}_{euc} \leq \frac{\zeta}{16C}.
    \end{align}
    Now for each $x \in S \setminus (\Omega\cup N)$ pick an $\rho_x \in (0, \min(r_x, \e_x))$, where $r_x$ is chosen as in the beginning of Step 2 (see the discussion leading up to \eqref{distance_assumption_step2_2d}). Then 
    \begin{align*}
       S \setminus (\Omega\cup N) \subset \bigcup_{x \in S \setminus (\Omega\cup N)} B_{\rho_x}(x). 
    \end{align*}
    The Besicovitch covering theorem states that there exist $C_2$ countable families $T_1, \ldots, T_{C_2} \subset S \setminus (\Omega\cup N)$ such that for each $j \neq j'$, it holds $T_j \cap T_{j'} = \emptyset$ and such that for each $x \neq x' \in T_j$, it holds
    \begin{align*}
        B_{\rho_{x}}(x) \cap B_{\rho_{x'}}(x') = \emptyset.
    \end{align*}
    Then, by the pigeonhole principle, there exists a $j \in \{1, \ldots, C_2\}$ such that 
    \begin{align}\label{disjoint_balls_big_measure}
        \mathcal{L}^2\Big(\bigcup_{x \in T_j} B_{\rho_x}(x)\Big) = \sum_{x \in T_j} \mathcal{L}^2( B_{\rho_x}(x)) \geq \frac{1}{C_2} \mathcal{L}^2(S \setminus (\Omega\cup N)) \geq \frac{\delta}{2C_2}.
    \end{align}
    Write $T_j = \{x_i\}_{i\in \N}$. Now fix an $i \in \N$.  
    Denote $x:= x_i \in T_j$ and write $\rho:= \rho_x$.
    Now, by the second step, more precisely \eqref{discontinuity_of_g_along_line2d}, there exists a $\varsigma_{x, \rho}>0$ such that for each $\varsigma \in (0, \varsigma_{x, \rho})$ there exists a $(\frac{1}{C}, C)$-bi-Lipschitz curve $\gamma = \gamma^{\rho/5C^2}:[0,\rho/5C^2] \to B_{\rho/5C}(x)$ and a set $A \subset [0,\rho/5C^2]$ such that $\mathcal{L}^1(A) \geq \frac{\rho\zeta}{20C^3}$ and for all $y \in \gamma(A)$, it holds
    \begin{align*}
        \norm{g(x)-g_\varsigma(y)}_{euc} \geq \frac{\zeta}{4C}.
    \end{align*}
    Using the locally uniform convergence of $g_\varsigma$ to $g$ on $U\setminus \Omega$, we may choose $\varsigma \in (0, \varsigma_{x, \rho})$ small enough such that additionally 
    \begin{align*}
        \norm{g(y)-g_\varsigma(y)}_{euc} \leq \frac{\zeta}{16C}, \ \mathrm{for\ all\ } y \in U \setminus\Omega,
    \end{align*}
    implying 
     \begin{align*}
        \norm{g(x)-g(y)}_{euc} \geq \frac{\zeta}{8C}, \ \mathrm{for\ all\ } y \in \gamma(A).
    \end{align*}
    Recalling \eqref{g_continuous_away_omega2d}, this implies that $\gamma(A) \subset \Omega \cap B_\rho(x)$. In particular, $O_\gamma:=\gamma^{-1}(\Omega) \subset [0,\rho]$ is an open set with $\mathcal{L}^1(O_\gamma) \geq \frac{\rho\zeta}{20C^3}$. Now, we can find a compact set $K \subset [0,1]$ that is the finite union of compact intervals $I_1,\ldots, I_n$ such that $K \subset O_\lambda$ and $\mathcal{L}^1(K) \geq \frac{\rho\zeta}{40C^3}$.

    Using mollification, we can find a function $\Tilde{f} \in C_c^\infty(\R^2)$ such that $\Tilde{f} \equiv 1$ on $\gamma(K)$ and 
    \begin{align}\label{replace_by_smooth_function}
        \norm{f-\Tilde{f}}_{W^{1,1}(B_{\rho}(x))} \leq 2^{-i-2}\epsilon.
    \end{align}
    By translating, we may assume that $x =0$. Now consider $F(x) := \Tilde{f}(\rho x) \in C_c^\infty(\R^2)$. Using that $\rho \leq 1$, we immediately get that 
    \begin{align}\label{rescale_smooth_function}
        \norm{F}_{W^{1,1}(B_1(0))} \leq \frac{1}{\rho^2} \norm{\Tilde{f}}_{W^{1,1}(B_\rho(0))}.
    \end{align}
    Moreover define the curve $\boldsymbol{\gamma}:[0, \frac{1}{5C^2}] \to B_{1/5C}(0)$ by $\boldsymbol{\gamma}(t) = \frac{1}{\rho}\gamma(\rho t)$. Then $\boldsymbol{\gamma}$ is $(\frac{1}{C}, C)$-bi-Lipschitz continuous.
    By construction, we have that $F$ is a $C_c^\infty$-function that satisfies $F(\boldsymbol{\gamma}(\frac{1}{\rho}K)) = \{1\}$. Using that $\mathcal{L}^1(\frac{1}{\rho}K) \geq \frac{\zeta}{40C^3}$, we may apply Lemma \ref{2dcapacity_of_line_segments}, to get that 
    \begin{align}
        \int_{B_{1}(0)} |F| + |DF| \di y \geq \frac{\zeta}{320C^5}.
    \end{align}
    Combining this with \eqref{rescale_smooth_function} and \eqref{replace_by_smooth_function}, we get that 
    \begin{align}\label{integral_estimate_on_specific_ball}
        \int_{B_{\rho}(0)} |f| + |Df| \di x \geq \frac{\rho^2\zeta}{320C^5} -2^{-i-2}\epsilon
    \end{align}
    Using that $\rho^2 = \mathcal{L}^2(B_\rho(0))/\pi$ together with \eqref{disjoint_balls_big_measure} and \eqref{integral_estimate_on_specific_ball} we get that 
    \begin{align}\label{lower_bound_w11}
        \int_{\R^2} |f| + |Df| \di y \geq \sum_{x \in T_j} \int_{B_{\rho_x}(x)} |f| + |Df| \di y \geq  \frac{\zeta}{320\pi C_2C^5}\mathcal{L}^2(S \setminus (\Omega\cup N)) - \epsilon.
    \end{align}
    Now note that $2$ is the Sobolev conjugate of $p=1$ in $\R^2$. 
    We now distinguish two cases: \\
    \textbf{Case 1.} 
    \begin{align}\label{condition_case_1}
         \int_{\R^2} |f| \di y \geq \frac{\zeta}{640\pi C_2C^5}\mathcal{L}^2(S \setminus (\Omega\cup N)) - \epsilon.
    \end{align}
    Then, by the H\"older inequality, we get that 
    \begin{align}
        \norm{f}_{L^2(\R^2)} \geq \norm{f}_{L^2(U)} \geq  \frac{1}{\sqrt{\mathcal{L}^2(U)}} \Big(\frac{\zeta}{640\pi C_2C^5}\mathcal{L}^2(S \setminus (\Omega\cup N)) - \epsilon \Big).
    \end{align}
    Then, by the Sobolev embedding theorem, we get that there exists a constant $c_{1,2}>0$ such that 
    \begin{align}\label{conclusion_case_1}
        \int_{\R^2} |Df| \di y \geq \frac{c_{1,2}}{\sqrt{\mathcal{L}^2(U)}} \Big(\frac{\zeta}{640\pi C_2C^5}\mathcal{L}^2(S \setminus (\Omega\cup N)) - \epsilon \Big).
    \end{align}
  \textbf{Case 2.}  
      \begin{align}\label{condition_case_2}
         \int_{\R^2} |f| \di y < \frac{\zeta}{640\pi C_2C^5}\mathcal{L}^2(S \setminus (\Omega\cup N)) - \epsilon.
    \end{align}
    Then \eqref{lower_bound_w11} implies that
    \begin{align}\label{conclusion_case2}
         \int_{\R^2} |Df| \di y \geq \frac{\zeta}{640\pi C_2C^5}\mathcal{L}^2(S \setminus (\Omega\cup N)) - \epsilon.
    \end{align}
    As one of case 1 and case 2 must be true, we get that 
    \begin{align}\label{final_bound_norm_df}
        \int_{\R^2} |Df| \di y \geq  \min\Big(1, \frac{c_{1,2}}{\sqrt{\mathcal{L}^2(U)}}\Big)\cdot\Big(\frac{\zeta}{640\pi C_2C^5}\mathcal{L}^2(S \setminus (\Omega\cup N)) - \epsilon\Big).
    \end{align}
    Now this contradicts \eqref{set_eps}, \eqref{Capacity_omega}, and \eqref{remaining_S} and hence proves the proposition.
\end{proof}

\begin{remark}
    Note that Proposition \eqref{2d_metric_speed_sobolev} answers Question \ref{q1} completely as we now know that $g, g^{-1} \in L^\infty_{\rm loc}$ is not sufficient (see \cite{de1995lip, ryborz2026infinitesimal}) but any additional Sobolev assumption gives positive answer. 
\end{remark}
The argument in the third step of the proof of Proposition \ref{2d_metric_speed_sobolev} yields another uniform lower $1$-capacity bound:
\smallskip
\begin{corollary}
    Let $U \subset \R^2$ be open and bounded and $\Omega \subset U$ be a Borel set such that the following holds: There exists $L>1$ and $\alpha \in (0,1)$ such that for every $x\in U$ there exists an $r_x>0$ such that for any $r \in (0,r_x)$ there exists a $(\frac{1}{L}, L)$-bi-Lipschitz curve $\gamma_r:[0,1] \to B_r(x)$ and a Borel set $A_{r, x} \subset [0,1]$ with $\mathcal{L}^1(A_{r, x}) \geq \alpha$ such that $\gamma_{r, x}(A_{r, x}) \subset \Omega$. Then there exists a constant $C=C(U, L, \alpha)$ such that $\Capa_1(\Omega)\geq C$. 
\end{corollary}
We are now able to state the main result answering Question \ref{q2}. 
\smallskip
\begin{theorem}\label{2d_infinitesimal_hilbertianity}
     Let $M$ be a $2$-manifold and $g$ be a Riemannian metric such that $g, g^{-1}\in W^{1, p}_{\rm loc} \cap L^\infty_{\rm loc}$. Assume furthermore that \eqref{condition_on_measure_bounded_compression} and Assumption \ref{assumption_on_mfd} and  Then $W^{1,2}_w(M) = H^2_1(M)$ and in particular, $(M, \sfd_g, \vol_g)$ is infinitesimally Hilbertian. 
\end{theorem}
\begin{proof}[Sketch of the proof.]
   We will not provide the full proof, as with the above ingredients, it follows the strategy of \cite{mondino2025equivalence}. The main steps are the following:
   \begin{itemize}
       \item For $f \in C^1$, we have that $|Df|_w = |\nabla_g f|_g$ $\vol_g$-almost everywhere. This follows exactly as in \cite[Proposition 4.24]{mondino2025equivalence}, using Proposition \ref{2d_metric_speed_sobolev} and localising the test plan around a Lebesgue point of $g$.  
       \item Following the argument in \cite[Proposition 4.25 and Lemma 4.26]{mondino2025equivalence}, we can infer that $f\in W^{1,2}_w(M)$ if and only if $f\in W^{1,2}_{\rm loc}(M)$ and $|\nabla_gf|_g \in L^2(M)$. In particular, we again get that for $f\in W^{1,2}_{\rm loc}(M)$, it holds $|Df|_w = |\nabla_g f|_g$ $\vol_g$-almost everywhere.
       \item Infinitesimal Hilbertianity then follows as in \cite[Corollary 4.27]{mondino2025equivalence}, from the linearity of differentiation and the fact that $g$ is an inner product on the tangent space. 
   \end{itemize}
\end{proof}

The case $d\geq 3$ is discussed in \cite{zhong2026scoop}.

\section*{Statements and Declarations}
\textbf{Conflict of interest statement.} The author has no relevant financial or non-financial
interests to disclose. The author has no conflict of interest to declare that are relevant to the content of this article.

\textbf{Data Availability.} Data sharing is not applicable to this article as no new data were created or analysed in this study.

\phantomsection
\addcontentsline{toc}{section}{References}
\bibliography{bibliographie}

@article{grant2011synthetic,
  title={Synthetic geometry and generalised functions},
  author={Grant, James DE},
  journal={Novi Sad Journal of Mathematics},
  volume={41},
  pages={75--84},
  year={2011}
}

@article{zhong2026scoop,
author = {Zhong, Deguang},
year = {2026},
month = {09},
pages = {},
title = {The Codimension-one Sobolev threshold for infinitesimal {H}ilbertianity of low-regularity Riemannian metrics},
doi = {10.13140/RG.2.2.14247.51363}
}

@article{steinbauer2009geroch,
  title={On the {G}eroch--{T}raschen class of metrics},
  author={Steinbauer, Roland and Vickers, James A},
  journal={Classical and Quantum Gravity},
  volume={26},
  number={6},
  pages={065001},
  year={2009},
  publisher={IOP Publishing}
}

@book{evans2025measure,
  title={Measure theory and fine properties of functions},
  author={Evans, Lawrence C and Gariepy, Ronald F},
  year={2025},
  publisher={Chapman and Hall/CRC}
}

@article{ryborz2026infinitesimal,
  title={The Infinitesimal Structure of Manifolds With Non-continuous {R}iemannian Metrics: V. Ryborz},
  author={Ryborz, Vanessa},
  journal={The Journal of Geometric Analysis},
  volume={36},
  number={9},
  pages={313},
  year={2026},
  publisher={Springer}
}

@article{ambrosio1990metric,
  title={Metric space valued functions of bounded variation},
  author={Ambrosio, Luigi},
  journal={Annali della Scuola Normale Superiore di Pisa-Classe di Scienze},
  volume={17},
  number={3},
  pages={439--478},
  year={1990}
}

@Article{burtscher2012length,
 Author = {Burtscher, Annegret Y.},
 Title = {Length structures on manifolds with continuous {Riemannian} metrics},
 FJournal = {The New York Journal of Mathematics},
 Journal = {New York J. Math.},
 ISSN = {1076-9803},
 Volume = {21},
 Pages = {273--296},
 Year = {2015},
 Language = {English},
 zbMATH = {6456672},
 Zbl = {1342.53054}
}

@Article{norris1997heat,
 Author = {Norris, James R.},
 Title = {Heat kernel asymptotics and the distance function in {Lipschitz} {Riemannian} manifolds},
 FJournal = {Acta Mathematica},
 Journal = {Acta Math.},
 ISSN = {0001-5962},
 Volume = {179},
 Number = {1},
 Pages = {79--103},
 Year = {1997},
 Language = {English},
 DOI = {10.1007/BF02392720},
 zbMATH = {1143394},
 Zbl = {0912.58041}
}

@article{de1988distanza,
  title={Distanza intrinseca su una varieta {R}iemanniana di {L}ipschitz},
  author={De Cecco, Giuseppe and Palmieri, Giuliana},
  journal={Rendiconti del seminario matematico Univers. Politecn. Torino},
  volume={46},
  pages={157--170},
  year={1988},
  publisher={Editrice universitaria Leyrotto e Bella: via Pigafetta 2/E, I 10129 Turin~…}
}

@article{de1991integral,
  title={Integral distance on a {L}ipschitz {R}iemannian manifold},
  author={De Cecco, Giuseppe and Palmieri, Giuliana},
  journal={Mathematische Zeitschrift},
  volume={207},
  number={1},
  pages={223--243},
  year={1991},
  publisher={Springer}
}

@article{de1995lip,
  title={LIP manifolds: from metric to Finslerian structure},
  author={De Cecco, Giuseppe and Palmieri, Giuliana},
  journal={Mathematische Zeitschrift},
  volume={218},
  number={1},
  pages={223--237},
  year={1995},
  publisher={Springer}
}

@article{de1990conversazioni,
  title={Conversazioni di matematica},
  author={De Giorgi, E},
  journal={Quaderni Dip. Mat. Univ. Lecce},
  volume={2},
  year={1990}
}

@article{di2021infinitesimal,
  title={Infinitesimal {H}ilbertianity of Locally {{\(\mathrm{CAT}(\kappa)\)}}-Spaces},
  author={Di Marino, Simone and Gigli, Nicola and Pasqualetto, Enrico and Soultanis, Elefterios},
  journal={The Journal of Geometric Analysis},
  volume={31},
  pages={7621--7685},
  year={2021},
  publisher={Springer}
}

@article{le2023universal,
  title={Universal infinitesimal {H}ilbertianity of sub-{R}iemannian manifolds},
  author={Le Donne, Enrico and Lu{\v{c}}i{\'c}, Danka and Pasqualetto, Enrico},
  journal={Potential Analysis},
  volume={59},
  number={1},
  pages={349--374},
  year={2023},
  publisher={Springer}
}

@article{lytchak2026infinitesimally,
  title={Infinitesimally {E}uclidean metric surfaces},
  author={Lytchak, Alexander and Soultanis, Elefterios},
  journal={arXiv preprint arXiv:2610.24626},
  year={2026}
}

@article{geroch1987strings,
  title={Strings and other distributional sources in general relativity},
  author={Geroch, Robert and Traschen, Jennie},
  journal={Physical Review D},
  volume={36},
  number={4},
  pages={1017},
  year={1987},
  publisher={APS}
}

@article{hajlasz1996sobolev,
  title={{S}obolev spaces on an arbitrary metric space},
  author={Haj{\l}asz, Piotr},
  journal={Potential Analysis},
  volume={5},
  pages={403--415},
  year={1996},
  publisher={Springer}
}

@article{luvcic2020infinitesimal,
  title={Infinitesimal {H}ilbertianity of weighted {R}iemannian manifolds},
  author={Lu{\v{c}}i{\'c}, Danka and Pasqualetto, Enrico},
  journal={Canadian Mathematical Bulletin},
  volume={63},
  number={1},
  pages={118--140},
  year={2020},
  publisher={Canadian Mathematical Society}
}

@article{kirchheim1994rectifiable,
  title={Rectifiable metric spaces: local structure and regularity of the {H}ausdorff measure},
  author={Kirchheim, Bernd},
  journal={Proceedings of the American Mathematical Society},
  volume={121},
  number={1},
  pages={113--123},
  year={1994}
}

@article {gigli2015differential,
    AUTHOR = {Gigli, Nicola},
     TITLE = {On the differential structure of metric measure spaces and
              applications},
  JOURNAL = {Memoirs of the American Mathematical Society},
    VOLUME = {236},
      YEAR = {2015},
    NUMBER = {1113},
     PAGES = {vi+91},
      ISSN = {0065-9266,1947-6221},
      ISBN = {978-1-4704-1420-7},
   MRCLASS = {53C23 (30L05 30L10 49Q15 58C20)},
  MRNUMBER = {3381131},
MRREVIEWER = {Davide\ Vittone},
       DOI = {10.1090/memo/1113},
       URL = {https://doi.org/10.1090/memo/1113},
}

@article{gigli2014overview,
  title={An overview of the proof of the splitting theorem in spaces with non-negative {R}icci curvature},
  author={Gigli, Nicola},
  journal={Analysis and Geometry in Metric Spaces},
  volume={2},
  number={1},
  pages={20141006},
  year={2014},
  publisher={De Gruyter Open}
}

@article{teleman1983index,
  title={The index of signature operators on {L}ipschitz manifolds},
  author={Teleman, Nicolae},
  journal={Publications Math{\'e}matiques de l'IH{\'E}S},
  volume={58},
  pages={39--78},
  year={1983}
}

@article{graf2020singularity,
	author = {Graf, Melanie},
	journal = {Communications in Mathematical Physics},
	number = {2},
	pages = {1417--1450},
	publisher = {Springer},
	title = {Singularity Theorems for ${C}^1$-{L}orentzian Metrics},
	volume = {378},
	year = {2020}}

@article{ambrosio2014inventio,
	author = {Ambrosio, Luigi and Gigli, Nicola and Savar{\'e}, Giuseppe},
	journal = {Inventiones mathematicae},
	number = {2},
	pages = {289--391},
	publisher = {Springer},
	title = {Calculus and heat flow in metric measure spaces and applications to spaces with {R}icci bounds from below},
	volume = {195},
	year = {2014}}

@article{ambrosio2014duke,
	author = {Ambrosio, Luigi and Gigli, Nicola and Savar{\'e}, Giuseppe},
	journal = {Duke Mathematical Journal},
	number = {7},
	pages = {1405--1490},
	publisher = {Duke University Press},
	title = {Metric measure spaces with {R}iemannian {R}icci curvature bounded from below},
	volume = {163},
	year = {2014}}

@book{grosser2013geometric,
	author = {Grosser, Michael and Kunzinger, Michael and Oberguggenberger, Michael and Steinbauer, Roland},
	publisher = {Springer Science \& Business Media},
	title = {Geometric theory of generalized functions with applications to general relativity},
	volume = {537},
	year = {2013}}

@article{gigli2018nonsmooth,
	author = {Gigli, Nicola},
	journal = {Memoirs of the American Mathematical Society},
	 VOLUME = {251},
        NUMBER = {1196},
	publisher = {American Mathematical Society},
	title = {Nonsmooth differential geometry--An approach tailored for spaces with {R}icci curvature bounded from below},
	year = {2018}}

@article{lott2009ricci,
	author = {Lott, John and Villani, C{\'e}dric},
	journal = {Annals of Mathematics},
	pages = {903--991},
	volume={169},
	publisher = {JSTOR},
	title = {{R}icci curvature for metric-measure spaces via optimal transport},
	year = {2009}}

@article{LottVillaniJFA,
	author = {Lott, John and Villani, C\'{e}dric},
	doi = {10.1016/j.jfa.2006.10.018},
	fjournal = {Journal of Functional Analysis},
	issn = {0022-1236,1096-0783},
	journal = {J. Funct. Anal.},
	mrclass = {53C21 (28A33 46E35 49Q20 53C23)},
	mrnumber = {2311627},
	mrreviewer = {Filippo\ Santambrogio},
	number = {1},
	pages = {311--333},
	title = {Weak curvature conditions and functional inequalities},
	url = {https://doi.org/10.1016/j.jfa.2006.10.018},
	volume = {245},
	year = {2007}}

@article{mondino2025equivalence,
  title={On the equivalence of distributional and synthetic {R}icci curvature lower bounds},
  author={Mondino, Andrea and Ryborz, Vanessa},
  journal={Journal of Functional Analysis},
  pages={111035},
  year={2025},
  publisher={Elsevier}
}

@article{saloff1992uniformly,
  title={Uniformly elliptic operators on {R}iemannian manifolds},
  author={Saloff-Coste, Laurent},
  journal={Journal of Differential Geometry},
  volume={36},
  number={2},
  pages={417--450},
  year={1992},
  publisher={Lehigh University}
}

@Article{sturm1995analysis,
 Author = {Sturm, Karl-Theodor},
 Title = {Analysis on local {Dirichlet} spaces. {II}: {Upper} {Gaussian} estimates for the fundamental solutions of parabolic equations},
 FJournal = {Osaka Journal of Mathematics},
 Journal = {Osaka J. Math.},
 ISSN = {0030-6126},
 Volume = {32},
 Number = {2},
 Pages = {275--312},
 Year = {1995},
 Language = {English},
 zbMATH = {880172},
 Zbl = {0854.35015}
}

@article{sturm2006geometryI,
	author = {Sturm, Karl-Theodor},
	title = {On the geometry of metric measure spaces.  {I}},
	journal={Acta Math.},
	volume={196},
	pages={65--131},
	year = {2006}}

@article{sturm2006geometryII,
	author = {Sturm, Karl-Theodor},
	title = {On the geometry of metric measure spaces. {I}{I}},
	journal={Acta Math.},
	volume={196},
	pages={133--177},
	year = {2006}}

@article{cheeger1999differentiability,
	author = {Cheeger, Jeff},
	journal = {Geometric \& Functional Analysis GAFA},
	pages = {428--517},
	publisher = {Birkh{\"a}user Verlag},
	title = {Differentiability of {L}ipschitz functions on metric measure spaces},
	volume = {9},
	year = {1999}}

@article{shanmugalingam2000newtonian,
  title={Newtonian spaces: an extension of {S}obolev spaces to metric measure spaces},
  author={Shanmugalingam, Nageswari},
  journal={Revista Matem{\'a}tica Iberoamericana},
  volume={16},
  number={2},
  pages={243--279},
  year={2000}
}

@article{sturm1997diffusion,
  title={Is a diffusion process determined by its intrinsic metric?},
  author={Sturm, Karl-Theodor},
  journal={Chaos, Solitons \& Fractals},
  volume={8},
  number={11},
  pages={1855--1860},
  year={1997},
  publisher={Elsevier}
}
\bibliographystyle{abbrv}
\end{document}